\documentclass[11pt]{article}
\usepackage{amsfonts, amssymb, amsmath, amsthm, xcolor, graphicx, pstricks, pstricks-add}
\begin{document}

\theoremstyle{plain}
\newtheorem{thm}{Theorem}[section]
\newtheorem{cor}[thm]{Corollary}
\newtheorem{con}[thm]{Conjecture}
\newtheorem{cla}[thm]{Claim}
\newtheorem{lm}[thm]{Lemma}
\newtheorem{prop}[thm]{Proposition}
\newtheorem{example}[thm]{Example}

\theoremstyle{definition}
\newtheorem{dfn}[thm]{Definition}
\newtheorem{alg}[thm]{Algorithm}
\newtheorem{prob}[thm]{Problem}
\newtheorem{rem}[thm]{Remark}

\newcommand{\E}{\textnormal{Ex}}

\renewcommand{\baselinestretch}{1.1}

\title{\bf A generalization of the extremal function of the Davenport-Schinzel sequences}
\author{
Cheng Yeaw Ku
\thanks{ Department of Mathematics, National University of
Singapore, Singapore 117543. E-mail: matkcy@nus.edu.sg} \and Kok
Bin Wong \thanks{
Institute of Mathematical Sciences, University of Malaya, 50603
Kuala Lumpur, Malaysia. E-mail:
kbwong@um.edu.my.} } \maketitle

\begin{abstract}\noindent
Let $[n]=\{1, \ldots, n\}$. A sequence $u=a_1a_2\dots a_l$ over $[n]$ is called $k$-sparse if $a_i = a_j$, $i > j$ implies $i-j\geq k$. In other words, every consecutive subsequence of $u$ of length at most $k$ does not have letters in common. Let $u,v$ be two sequences. We say that $u$ is $v$-free, if  $u$ does not contain a subsequence isomorphic to $v$. Suppose there are only $k$ letters appearing in $v$. The extremal function $\E(v,n)$ is defined as the maximum length of all the $v$-free and $k$-sparse sequences. In this paper, we study a generalization of the extremal function  $\E(v,n)$.
\end{abstract}

\bigskip\noindent
{\sc keywords:} Davenport-Schinzel sequences, pattern avoidance

\section{Introduction}

Let $[n]=\{1, \ldots, n\}$. An element in $[n]$ shall be called a \emph{letter}. An element $u$ is said to be a \emph {sequence} of \emph{length} $l$ over $[n]$, if $u=a_1a_2\dots a_l$, $a_i\in [n]$ for all $i$. We shall denote the length of $u$ by $\vert u\vert$. For each $i\in [n]$, let
\begin{equation}
\sigma_i(u)=\sum_{\substack{1\leq j\leq l,\\ a_j=i}} 1.\notag
\end{equation}
Note that $\sigma_i(u)$ is the number of times the letter $i$ appears in the sequence $u$. We shall denote the number of distinct letters that appear in $u$ by $\vert\vert u\vert\vert$. Note that
\begin{equation}
\vert\vert u\vert\vert=\vert \{ i\in [n]\ :\ \sigma_{i}(n)>0\}\vert,\notag
\end{equation}
where $\vert A\vert$ denote the number of elements in the set $A$.  For examples $\vert 121331\vert=6$, $\sigma_1(121331)=3$, and $\vert\vert 121331\vert\vert=3$.

Two sequences $u=a_1a_2\dots a_l$ and $v=b_1b_2\dots b_{l'}$ are said to be \emph{isomorphic}, if $l=l'$ and there is a bijection $\beta : [n] \to [n]$ such that $\beta (a_j)=b_j$ for all $j$. For example $121331$ and $242112$ are isomorphic. A sequence $u$ is said to be  \emph{normal}  over $[n]$, if $\sigma_i(u)>0$ for all $i\in [n]$, and
the first occurrences of $1,2,\dots, n$ in $u$, if we scan $u$ from left to right, come in this order. Let $\mathcal S(n)$ be the set of all sequences over $n$. Clearly, every element in $\mathcal S(n)$ is isomorphic to a unique normal sequence. For example $121331$ is normal but $242112$ is not.

A sequence $v$ is said to be a \emph{subsequence} of $u=a_1a_2\dots a_l$, if $v=a_{l_1}a_{l_2}\dots a_{l_j}$ where $1\leq l_1<l_2<\cdots <l_j\leq l$. It is called a
\emph{consecutive subsequence} if $l_{i+1}=l_{i}+1$ for all $1\leq i\leq j-1$. For example, $121331$ and $513435$ are subsequences of $144255134351$. Furthermore, $513435$ is a consecutive subsequence, but $121331$ is not.

A sequence $u=a_1a_2\dots a_l$ is called $k$-\emph{sparse} if $a_i = a_j$, $i > j$ implies $i-j\geq k$. In other words, for every consecutive subsequence $v$ of $u$ of length at
most $k$, we have $\sigma_i(v)\leq 1$ for all $i$.

Let $u$ and $v$ be two sequences. We say that $u$ contains $v$ and write $u\supset v$,  if $u$ has a
subsequence isomorphic to $v$. If $u$ does not contain $v$, we say that $u$ is $v$-\emph{free}. We shall denote the set of letters appearing in $u$ by $L(u)$, i.e.,
\begin{equation}
L(u)=\{ i\in [n]\ :\ \sigma_{i}(u)>0\}.\notag
\end{equation}

Let $\mathbb N$ and $\mathbb R$ be the sets of positive integers and real numbers, respectively. Let $f,g : \mathbb N\to \mathbb R$ be two functions. The asymptotic notation
$f\ll g$ is synonymous to the $f = O(g)$ notation and means that there exists a positive integer $n_0$ and a constant $c>0$ such that
$\vert f(n)\vert \leq  c\vert g(n)\vert$ holds for every $n \geq  n_0$. The
subscripts, such as $f\ll_k g$, indicate that $c$ depends only on the parameter $k$. We say $f,g$ have the same \emph{asymptotic order} if $f = O(g)$ and $g = O(f)$.

Let $v\in \mathcal S(n)$. We associate with $v$ the extremal functions
\begin{equation}
\E(v,n)=\max \left\{ \vert u\vert\ :\ \textnormal{$u$ is $v$-free and $\vert\vert v\vert\vert$-sparse, and $\vert\vert u\vert\vert\leq n$} \right \},\notag
\end{equation}
\begin{equation}
\E(v,k,n)=\max \left\{ \vert u\vert\ :\ \textnormal{$u$ is $v$-free and $k$-sparse, and $\vert\vert u\vert\vert\leq n$} \right \}.\notag
\end{equation}
Note that $\E(v,n)=\E(v,\vert \vert v\vert\vert,n)$. The extremal function $\E(v,k,n)$ was first introduced by Adamec, Klazar and Valtr \cite{AKV}, and it generalized the Davenport-Schinzel sequences \cite{DS}. Two surveys have been written by Klazar \cite{Klazar2, Klazar3} on this topic (see also \cite{CFZ,CK, JM, JMS, LP, MSS, Pettie1, Pettie2} for related results). We note here that $\E(v, k, n)$ is always well-defined. In fact, Klazar \cite{Klazar} showed that $\E(v, n)\ll_v n^2$, and Adamec, Klazar and Valtr \cite{AKV} showed that
\begin{equation}\label{relation_in}
\E(v,n)\ll_{v,k}\E(v,k,n)\leq \E(v,n).
\end{equation}

Note that a sequence $u=a_1a_2\dots a_l$ is $k$-sparse if and only if for each $j$, $1\leq j\leq l$, the subsequence
\begin{equation}
a_{j}a_{j+1}\dots a_{j+(k-1)},\notag
\end{equation}
does not have any letter in common, i.e., $\sigma_i(a_{j+1}a_{j+2}\dots a_{j+k})\leq 1$ for all $i$. This motivates us to consider $(k,r)$-\emph{sparse} sequences, that is a sequence $u=a_1a_2\dots a_l$ is called $(k,r)$-sparse if and only if for each $j$, $1\leq j\leq l$, the subsequence
\begin{equation}
a_{j}a_{j+r}\dots a_{j+(k-1)r},\notag
\end{equation}
does not have any letter in common. Let $v\in \mathcal S(n)$. We associate with $v$ the extremal functions
\begin{equation}
\E_r(v,n)=\max \left\{ \vert u\vert\ :\ \textnormal{$u$ is $v$-free and $(\vert\vert v\vert\vert,r)$-sparse, and $\vert\vert u\vert\vert\leq n$} \right \},\notag
\end{equation}
\begin{equation}
\E_r(v,k,n)=\max \left\{ \vert u\vert\ :\ \textnormal{$u$ is $v$-free and $(k,r)$-sparse, and $\vert\vert u\vert\vert\leq n$} \right \}.\notag
\end{equation}

Note that $\E_1(v,k,n)=\E(v,k,n)$.  In this paper, we study  the extremal function $\E_r(v,k,n)$. We shall always assume that $k\geq \vert\vert v\vert\vert$. This is because if $k< \vert\vert v\vert\vert$ and $n\geq k$, then $\E_r(v,k,n)=\infty$. For instance, the sequence $1^r2^r\dots k^r1^r2^r\dots k^r1^r2^r\dots $ is  $v$-free and $(k,r)$-sparse. Here
\begin{equation}
a^r=\overbrace{a\dots a}^{r\ \textnormal{times}}.\notag
\end{equation}

\section{Relations between $\E_r(v,k,n)$ and $\E_1(v,k,n)$}

Recall that we shall always assume $k\geq \vert\vert v\vert\vert$.

\begin{thm}\label{thm_relation1} $\E_1(v,(k-1)r+1,n)\leq  \E_r(v,k,n)\leq r\E_1(v,k,n)$.
\end{thm}

\begin{proof} The first inequality follows by noting that a $v$-free and $((k-1)r+1)$-sparse sequence $u$ with $\vert\vert u\vert\vert\leq n$ is also $v$-free and $(k,r)$-sparse.

Let $a_i\in [n]$, and $u=a_1a_2\dots a_l$ be a $v$-free and $(k,r)$-sparse sequence with $\vert\vert u\vert\vert\leq n$. Consider the subsequence
$w=b_1b_2\dots b_m$ of $u$, where $b_i=a_{1+(i-1)r}$ and $l-r+1\leq 1+(m-1)r\leq l$. Note that $w$ is $v$-free and $k$-sparse. Thus, $m\leq \E_1(v,k,n)$.
From $l-r+1\leq 1+(m-1)r$, we obtain $l\leq rm$. Therefore $l\leq r\E_1(v,k,n)$ and $\E_r(v,k,n)\leq r\E_1(v,k,n)$.
\end{proof}

The following corollary is a generalization of equation (\ref{relation_in}).

\begin{cor}\label{cor_inequality1} $\E_r(v,n)\ll_{r,v,k}\E_r(v,k,n)\leq \E_r(v,n)$.
\end{cor}

\begin{proof} By Theorem \ref{thm_relation1}, $\E_1(v,(k-1)r+1,n)\leq  \E_r(v,k,n)$. By equation (\ref{relation_in}), $\E_1(v,n)\ll_{r,v,k}\E_1(v,(k-1)r+1,n)$. It then follows from Theorem \ref{thm_relation1} that $\frac{\E_r(v,n)}{r}\leq \E_1(v,n)$. Therefore $\E_r(v,n)\ll_{r,v,k}\E_r(v,k,n)$. The second inequality is obvious.
\end{proof}

\begin{cor}\label{cor_order} $\E_r(v,k,n)$ and $\E_1(v,n)$ have the same asymptotic order.
\end{cor}

\begin{proof} The inequality $\E_r(v,k,n)=O(\E_1(v,n))$ follows from Theorem \ref{thm_relation1} and Corollary \ref{cor_inequality1}. The inequality $\E_1(v,n)=O(\E_r(v,k,n))$ follows from Theorem \ref{thm_relation1} and by noting that $\E_1(v,n)\ll_{r,v,k}\E_1(v,(k-1)r+1,n)$.
\end{proof}

\begin{cor}\label{cor_order2} $\E_r(v,k,n)=O(n^2)$.
\end{cor}

\begin{proof} Follows from Corollary \ref{cor_order} and \cite[Theorem 1.4]{Klazar}.
\end{proof}

Now that we have settled the asymptotic relation between $\E_r(v,k,n)$ and $\E_1(v,n)$. What about the relation between $\E_r(v,k,n)$ and $\E_1(v,k,n)$. By Theorem \ref{thm_relation1}, $\E_r(v,k,n)\leq r\E_1(v,k,n)$.

\begin{prob} When is $\E_r(v,k,n)=r\E_1(v,k,n)$?
\end{prob}

We shall give a sufficient condition for which the equality holds.

\begin{thm}\label{thm_relation2} If $v$ is $2$-sparse and $\vert\vert v\vert\vert\geq 2$, then $\E_r(v,k,n)=r\E_1(v,k,n)$.
\end{thm}

\begin{proof}
Let $b_i\in [n]$, and $w=b_1b_2\dots b_m$ be a $v$-free and $k$-sparse sequence with $\vert\vert w\vert\vert\leq n$ and $m=\E_1(v,k,n)$. Let
\begin{equation}
u=b_1^{r}b_2^{r}\dots b_{m}^{r}.\notag
\end{equation}
Note that $u$ is $(k,r)$-sparse. If $u$ is not $v$-free, then it contains a subsequence $b_{i_1}^{s_{i_1}}b_{i_2}^{s_{i_2}}\dots b_{i_t}^{s_{i_t}}$ isomorphic to $v$. Since $v$ is 2-sparse, $s_{i_1}=s_{i_2}=\cdots =s_{i_t}=1$. Then $w$ contains a subsequence $b_{i_1}b_{i_2}\dots b_{i_t}$ isomorphic to $v$, a contradiction. Hence $u$ is $v$-free and $rm\leq \E_r(v,k,n)$. The theorem then follows from Theorem \ref{thm_relation1}.
\end{proof}

\begin{cor}\label{cor_realation2}
\begin{equation}
\E_r(abab,k,n)=\begin{cases}
rn, &\textnormal{if $1\leq n\leq k-1$}\\
 r(2n-k+1), &\textnormal{if $n\geq k$}.
\end{cases}\notag
\end{equation}
\end{cor}

\begin{proof} Follows from Theorem \ref{thm_relation2} and \cite[Theorem 2.2]{Klazar}.
\end{proof}

A sequence $v$ over $[n]$ is called a \emph{chain} if $\sigma_{i}(v)\leq 1$ for all $i\in [n]$. Note that if $v$ is a chain, then $\E_1(v,k,n)=\min(\vert v\vert-1,n)$. Since a chain is 2-sparse, the following corollary follows from Theorem \ref{thm_relation2}.

\begin{cor}\label{cor_realation3} If $v$ is a chain, then $\E_r(v,k,n)=r\min(\vert v\vert-1,n)$.
\end{cor}

In general, $\E_r(v,k,n)\neq r\E_1(v,k,n)$. For instance, $\E_1(a^2,k,n)=n=\E_r(a^2,k,n)$, as the following theorem shows.

\begin{thm}\label{thm_pre_separable} If $n\geq rk$, then $\E_r(a^s,k,n)=(s-1)n$.
\end{thm}

\begin{proof} Let $u$ be a $a^s$-free and $(k,r)$-sparse sequence. Note that $\sigma_i(u)\leq s-1$ for all $i\in [n]$. Therefore $\E_r(a^s,k,n)\leq (s-1)n$. Let $v=12\dots n$. Then the sequence
\begin{equation}
v^{s-1}=\overbrace{v\dots v}^{\textnormal{$s-1$ times}},\notag
\end{equation}
is $a^s$-free and $(k,r)$-sparse. Furthermore $\vert v^{s-1}\vert=(s-1)n$. Hence $\E_r(a^s,k,n)\geq (s-1)n$ and the theorem follows.
\end{proof}

\section{Minimum value of $\E_r(v,k,n)$}

Let $a_i\in [n]$, and $u=a_1^{s_1}a_2^{s_2}\dots a_m^{s_m}$ be a sequence. A sequence $v$ is called a \emph{blow-up} of $u$ if $v=a_1^{s_1'}a_2^{s_2'}\dots a_m^{s_m'}$ and $s_i'\geq s_i$ for all $i$. Now, let us look at the minimum value of $\E_r(v,k,n)$. If $v$ is not a blow-up of a chain, then the sequence $1^r2^r\dots n^r$ is  $v$-free and $(k,r)$-sparse. Hence, $\E_r(v,k,n)\geq rn$.

\begin{prob} Suppose $v$ is not a blow-up of a chain. When is $\E_r(v,k,n)=rn$?
\end{prob}

\begin{lm}\label{lm_pre_thm1} If $n<k$ and $v$ is not a blow-up of a chain, then $\E_r(v,k,n)=rn$.
\end{lm}

\begin{proof} Suppose $\E_r(v,k,n)>rn$. Let $a_i\in [n]$, and $u=a_1a_2\dots a_l$ be a $v$-free and $(k,r)$-sparse sequence with $l=\E_r(v,k,n)$. Since $l\geq rn+1$, $u$ contains the subsequence $w=a_1a_{1+r}\dots a_{1+{n}r}$. Note that all the letters appearing in $w$ must be distinct, for $u$ is $(k,r)$-sparse. However this is impossible as we only have $n$ letters. Hence $\E_r(v,k,n)=rn$.
\end{proof}

\begin{lm}\label{lm_pre_thm2}  Suppose $n\geq k$ and $v$ is not a blow-up of a chain. Then $\E_1(v,k,n)=n$ if and only if $v=awa$, where $w$ is a chain that does not contain the letter $a$.
\end{lm}

\begin{proof} Suppose $\E_1(v,k,n)=n$. Note that the sequence $12\dots n1$ is $k$-sparse. Since $\vert 12\dots n1\vert=n+1$, it is not $v$-free. This implies that $v=1w1$, where $w$ is a chain that does not contain the letter $1$.

Let $u$ be a $v$-free and $k$-sparse sequence. If a letter, say $1$ appears more than once in $u$, then $u$ contains a subsequence $1b_1b_2\dots b_{k-1}1$ where all $b_i$ are distinct letters. Since $\vert\vert w\vert\vert\leq k-1$, $1b_1b_2\dots b_{k-1}1$ has an isomorphic copy of $v$, a contradiction. Therefore every element appears at most once, and $\vert u\vert\leq n$. Hence $\E_1(v,k,n)=n$.
\end{proof}

\begin{thm}\label{thm_part2} Suppose $n\geq k$ and $v$ is not a blow-up of a chain. If $\E_r(v,k,n)=rn$, then
 $v=awa$, where $w$ is a blow-up of a chain that does not contain the letter $a$.
\end{thm}

\begin{proof} By Lemma \ref{lm_pre_thm2}, we may assume that $r\geq 2$. Now the sequence $1^r2^r\dots n^r1$ is $(k,r)$-sparse. Since $\vert 1^r2^r\dots n^r1\vert=rn+1$, it is not $v$-free. This implies that $v$ is isomorphic to $1^sw1$, where $w$ is a subsequence of $2^r\dots n^r$ and $1\leq s\leq r$. Similarly from $\vert n1^r2^r\dots n^r\vert=rn+1$, we deduce that $v$ is isomorphic to $nw'n^{s'}$, where $w'$ is a subsequence of $1^r2^r\dots (n-1)^r$ and $1\leq s'\leq r$. Hence $v=awa$, where $w$ is a blow-up of a chain that does not contain the letter $a$.
\end{proof}

In general, the converse of Theorem \ref{thm_part2} does not hold. For instance, when $v=abba$, $a\neq b$, $r=2$ and $n\geq 2k$, the sequence $w= 12\dots (n-1)n12\dots (n-1)n^2$ is $v$-free and $(k,2)$-sparse, and $\vert w\vert=2n+1>2n=rn$.

Klazar \cite[Theorem 3.1]{Klazar} showed that
\begin{equation}\label{eq2}
\E_1(abba,k,n)=\begin{cases}
n, &\textnormal{if $1\leq n\leq k-1$}\\
2n+\left\lfloor \frac{n-1}{k-1}\right\rfloor-1, &\textnormal{if $n\geq k$}.
\end{cases}
\end{equation}
It would be interesting to find the exact value of $\E_r(abba,k,n)$ for $r\geq 2$.

\begin{lm}\label{lm_pre_exact_value} Let $u$ be a $(k,r)$-sparse sequence over $[n]$ with $k\geq 2$ and $\vert u\vert=2r$. Then $\sigma_i(u)\leq r$ for all $i\in [n]$.
\end{lm}

\begin{proof} Let $u=a_1a_2\dots a_{2r}$ where  $a_i\in [n]$. Since the sequence is $(k,r)$-sparse and $k\geq 2$, $a_j$ and $a_{j+r}$ are distinct letters for $1\leq j\leq r$. By the pigeonhole principle, $\sigma_i(u)\leq r$ for all $i$.
\end{proof}

\begin{thm}\label{thm_exact_value} If $r\geq 5$, then $\E_r(abba,k,n)=rn$. Furthermore, the longest sequence realizing this length is $1^r2^r\dots n^r$.
\end{thm}

\begin{proof} First, we assume $k=2$. For $n=1$, the  $abba$-free and $(2,r)$-sparse normal sequences with exactly one letter are
\begin{itemize}
\item[(i)] $1^j, \ \ 1\leq j\leq r$.
\end{itemize}
For $n=2$, the  $abba$-free and $(2,r)$-sparse normal sequences  with exactly two letters are
\begin{itemize}
\item[(i)] $1^i2^j, \ \ 1\leq i,j\leq r$;
\item[(ii)] $1^i21^j, \ \ 1\leq i,j\leq r$, $i+j+1\leq r$;
\item[(iii)] $1^i212^j, \ \ 1\leq i,j\leq r$, $i+j+2\leq 2r-2$.
\end{itemize}
Thus the theorem holds for $n=1,2$. Suppose $n\geq 3$. Assume that the theorem holds for all $n'$ with $1\leq n'<n$.

Let $u$ be a $abba$-free and $(2,r)$-sparse sequence with $\vert u\vert =\E_r(abba,2,n)$. We may write
\begin{equation}
u=a_1a_2\dots a_ra_{r+1}a_{r+2}\dots a_{2r}w,\notag
\end{equation}
where $a_i\in [n]$ and $w$ is a sequence. We may assume that $1=a_1=a_2=\cdots=a_s$, $a_{s+1}\neq 1$ and $1\leq s\leq r$.

\noindent
{\bf Case 1.} Suppose $s<r$. Consider the sequence $1u$. Note that $1u$ is $abba$-free. By the choice of $u$, $1u$ is not $(2,r)$-sparse. This implies that $a_r=1$. Therefore $u$ must be of the form
\begin{equation}
1^sa_{s+1}\dots a_{r-1}1w_1,\notag
\end{equation}
where $w_1=a_{r+1}a_{r+2}\dots a_{2r}w$ and $a_{r+1}\neq 1$ for $a_1=1$.

Suppose $1\in L(w_1)$. Then $w_1=v1v'$ for some sequences $v,v'$ and $\vert v\vert\geq 1$. If $a_{s+1}\in L(v)$, then $u$ contains the subsequence $1a_{s+1}^21$, a contradiction. If $a_{s+1}\in L(v')$, then $u$ contains the subsequence $a_{s+1}1^2a_{s+1}$, a contradiction.
Therefore $a_{s+1}\notin L(w_1)$ and $1\leq \vert\vert w_1\vert\vert\leq n-1$. Note that $w_1$ is also $abba$-free and $(2,r)$-sparse.
By induction, $\vert w_1\vert\leq r(n-1)$. Suppose $\vert w_1\vert= r(n-1)$. Then $w_1$ must be of the form $b_1^rb_2^r\dots b_{n-1}^r$ and $\{b_1,b_2,\dots ,b_{n-1}\}=[n]\setminus \{a_{s+1}\}$. One of the letters $b_2,b_3,\dots, b_{n-1}$ must be 1, for $b_1\neq 1$. So $u$ contains the subsequence $1b_1^21$, a contradiction.
Hence $\vert w_1\vert< r(n-1)$ and $\vert u\vert=r+\vert w_1\vert<rn$.

Suppose $1\notin L(w_1)$. Again,  by induction, $\vert w_1\vert\leq r(n-1)$. Suppose $\vert w_1\vert= r(n-1)$. Then $w_1$ must be of the form $b_1^rb_2^r\dots b_{n-1}^r$ and $\{b_1,b_2,\dots ,b_{n-1}\}=[n]\setminus \{1\}$. One of the letters $b_1,b_2,b_3,\dots, b_{n-1}$ must be $a_{s+1}$. If $b_1=a_{s+1}$, then $a_{r+1}=\cdots =a_{2r}=a_{s+1}$, and $1^sa_{s+1}\dots a_{r-1}1a_{s+1}^r$ is a consecutive subsequence of length $2r$, contradicting Lemma \ref{lm_pre_exact_value}.

If $b_i=a_{s+1}$ for some $i\geq 2$, then $u$ contains the subsequence $a_{s+1}b_1^2a_{s+1}$, a contradiction. Hence $\vert w_1\vert< r(n-1)$ and $\vert u\vert=r+\vert w_1\vert<rn$.

\noindent
{\bf Case 2.} Suppose $s=r$. The $u$ is of the form
\begin{equation}
1^ra_{r+1}\dots a_{2r}w.\notag
\end{equation}
By Lemma \ref{lm_pre_exact_value}, $1\notin L(a_{r+1}\dots a_{2r})$.

Suppose $1\notin L(w)$. By induction, $\vert w\vert\leq r(n-1)$. Suppose $\vert w\vert= r(n-1)$. Then $w$ must be of the form $2^r3^r\dots n^r$. Hence $\vert u\vert=rn$ and $u=1^r2^r\dots n^r$. If $\vert w\vert< r(n-1)$, then $\vert u\vert<rn$.

Suppose $1\in L(w)$. Then all the letters $1,a_{r+1},\dots ,a_{2r}$ are distinct.

Suppose $w$ does not any contain consecutive subsequence isomorphic to $e^2$. Then the sequence $1a_{r+1}\dots a_{2r}w$ is $abba$-free and 2-sparse. By (\ref{eq2}), $\vert 1a_{r+1}\dots a_{2r}w\vert\leq 3n-2$. So, $\vert u\vert \leq 3n-2+(r-1)<rn$, for $n\geq 3$ and $r\geq 5$.

Suppose $w$ contains a consecutive subsequence isomorphic to $e^2$. We may assume that $w$ is of the form
\begin{equation}
w_1c^2w_2,\notag
\end{equation}
where $w_1,w_2$ are sequences, $c$ is a letter, $\vert w_1\vert\geq 0$ and $\vert w_2\vert\geq 0$. Furthermore, if $\vert w_1\vert >0$, we may assume that $w_1$ does not contain any consecutive subsequence isomorphic to $e^2$. Note that
\begin{equation}
L(1^ra_{r+1}\dots a_{2r}w_1)\cap L(w_2)\subseteq \{c\}.\notag
\end{equation}
Let $L(1^ra_{r+1}\dots a_{2r}w_1)\setminus \{c\}=t_1$ and $L(w_2)\setminus \{c\}=t_2$. Then $t_1+t_2+1=n$. Since the sequence $1a_{r+1}\dots a_{2r}w_1$ is $abba$-free and 2-sparse, by (\ref{eq2}), $\vert 1a_{r+1}\dots a_{2r}w_1\vert\leq 3(t_1+1)-2=3t_1+1$. Thus $\vert 1^ra_{r+1}\dots a_{2r}w_1\vert\leq 3t_1+r$.

Note that $t_1\geq r$, for $a_{r+1},\dots ,a_{2r}$ are distinct. By induction, $\vert c^2w_2\vert\leq r(t_2+1)=r(n-t_1)$. Therefore $\vert u\vert\leq rn-(r-3)t_1+r\leq rn-(r-3)r+r=rn-r^2+4r=rn-r(r-4)<rn$, for $r\geq 5$.

So, the theorem holds for $k=2$. For $k>2$, the theorem follows by noting that $\E_r(abba,k,n)\leq \E_r(abba,2,n)$.
\end{proof}

Theorem \ref{thm_exact_value} leads us to believe that the following conjecture is true.

\begin{con}\label{con_min} Let $t_1,t_2,\dots ,t_s$ be integers and $t=\max(t_1,t_2,\dots, t_s)$. There exists a constant $C$ depending on $t$, such that for every $r\geq C$,
\begin{equation}
\E_r(12^{t_1}3^{t_2}\dots s^{t_{s-1}}(s+1)^{t_s}1,k,n)=rn.\notag
\end{equation}
\end{con}

\end{document}